\documentclass[11pt,a4paper,twoside]{article}
\author{Stephen Bedford}
\title{Global minimisers of cholesteric liquid crystal systems}

\usepackage{amsmath,amsfonts,euscript,amssymb,amsthm,times,graphicx,bm,bbm,fancyhdr,color}
\usepackage{txfonts,epsfig,subfig,mathrsfs,enumerate,textcomp,wrapfig,float,sidecap}
\usepackage[margin=1.0in]{geometry}
\usepackage[showonlyrefs]{mathtools}
\theoremstyle{plain}
\newtheorem{theorem}{Theorem}
\newtheorem{proposition}[theorem]{Proposition}
\newtheorem{definition}[theorem]{Definition}
\newtheorem{lemma}[theorem]{Lemma}
\newtheorem{corollary}[theorem]{Corollary}
\newtheorem{remark}[theorem]{Remark}

\newtheorem*{theorem*}{Theorem}
\newtheorem*{proposition*}{Proposition}
\newtheorem*{definition*}{Definition}
\newtheorem*{lemma*}{Lemma}
\newtheorem*{corollary*}{Corollary}
\newtheorem*{remark*}{Remark}

\newcommand{\vn}{{\bf n}}
\newcommand{\vv}{{\bf v}}
\newcommand{\vm}{{\bf m}}

\newcommand{\vphi}{\boldsymbol\phi}

\renewenvironment{proof}{{\bf{Proof}\vspace{3 mm}}}{\qed}

\begin{document}

\maketitle
\mathtoolsset{showonlyrefs=true}

\begin{abstract}

In this paper we examine the modelling and minimisation of cholesteric liquid crystals systems within the Oseen-Frank theory. We focus on a cuboid domain with the frustrated boundary conditions $\vn(0,x,y)=(1,0,0)$ and $\vn(1,x,y)=(0,0,1)$. With general elastic constants, we find the unique global minimisers amongst functions of one variable and prove that these are global minimisers of the entire problem if the cholesteric pitch is sufficiently long. Finally we show that our analysis easily translates over the study the global stability of the constant state $\vn(x,y,z) = (0,0,1)$ with unfrustrated boundary conditions.

\end{abstract}

%
%
%
%
%
%

\section{Introduction}\label{section: intro}

The mathematical study of liquid crystals through the calculus of variations has received increasing attention in recent decades. In this area, the main supposition is that the free energy of the systems can be modelled by an integral functional. Unfortunately there is no broad consensus as to the correct form of the free energy density or even what the order parameter should be. The Oseen-Frank theory uses the anisotropic axis of the molecules, the director, to model the energy \cite{oseen1933theory,frank1958theory}. Ericksen's theory uses the director and a scalar order parameter which measures the variance from the director \cite{ericksen1991liquid}. The Landau-de Gennes and Q-Tensor theories use matrix order parameters, which are symmetric and traceless, to properly respect the head-to-tail symmetry of the constituent molecules \cite{de1995physics,ball2010nematic}. In general there has been a great deal of attention on how to best model liquid crystal systems. However there are far fewer mathematical studies of these theories for a given setup. The works of Majumdar, Robbins, Zyskin and Newton are notable exceptions \cite{majumdar2007topology,majumdar2004elastic}. In this paper we use the Oseen-Frank free energy functional and study two specific problems. We will focus on cholesteric liquid crystals because they are more general, and less well understood, than nematic liquid crystals. We consider a liquid crystal sample held in a general cuboid domain     
\begin{equation}\label{0.1}
 \Omega = (-L_1,L_1)\times(-L_2,L_2)\times(0,L_3).
\end{equation}
Then we shall investigate global minimisers with two different sets of boundary conditions on the top and bottom plates. The first has homeotropic boundary conditions on both faces while the second uses planar non-degenerate boundary conditions on the bottom face instead, so that $\vn(x,y,0) = {\bf e}_1$. On the other four faces we will simply impose periodic boundary conditions. To the best of our knowledge, the frustrated boundary condition case has not been studied rigorously with the vectorial calculus of variations before. The other case involving homeotropic anchoring was considered by Gartland et al. \cite{gartland2010electric} in a purely one-dimensional formulation. In contrast, the results that we will prove in Section \ref{section: second problem} apply in three spatial dimensions but they still qualitatively agree with the conclusions from \cite{gartland2010electric}. An important aspect of the Oseen-Frank model that we shall consider is the assumptions on the elastic constants, $K_i$. The general free energy density is given by 
\begin{equation}\label{0.2}
 w(\vn,\nabla \vn) := K_1 \left( \nabla \cdot \vn \right)^2 + K_3 \left( \vn \cdot \nabla \times \vn +t \right)^2 +K_3 \left|\vn \times \nabla \times \vn \right|^2 + \left( K_2 + K_4 \right) \left( \rm{tr}\left( \nabla \vn^2 \right) - \left( \nabla \cdot \vn \right)^2 \right).
\end{equation}
Naturally the values of the elastic constants are an important factor in determining the minimisers. The usual assumption made in the literature, see for example \cite{stewart2004static,virga1994variational}, is that it should be the case that $w(\vn,\nabla \vn)\geqslant 0$, and equals zero if and only if $\vn$ is any fixed rotation of 
\begin{equation}\label{0.3}
 \left(\begin{array}{c} \cos(tz) \\ \sin(tz) \\ 0 \end{array} \right).
\end{equation}
This logically leads to the following set of inequalities, called Ericksen's inequalities \cite{ericksen1966inequalities,virga1994variational},
\begin{equation}\label{0.4}
 \begin{split}
  K_1,K_2,K_3 \geqslant 0,\,\, K_2+K_4 = 0\,\, {\rm{if}}\,\, t\neq 0  \\
  2K_1 \geqslant K_2+K_4,\,\, K_2 \geqslant |K_4|,\,\, K_3 \geqslant 0 \,\, {\rm{if}} \,\, t =0.
 \end{split}
\end{equation}
This logic has possible flaws for two reasons. The first is that it is slightly illogical from a mathematical point of view to derive an integral functional assuming a given minimum of the Lagrangian; it is preferable for the minimum state to emerge naturally from the form of the free energy density. Secondly the two sets of inequalities in \eqref{0.4} do not exhibit a continuous dependence on the cholesteric twist parameter $t$. Nematic liquid crystals are, in some sense, a limit of cholesteric liquid crystals as $t\rightarrow 0$, hence one would expect a set of inequalities which would respect this limit. The restriction that $K_2+K_4=0$, so that we have no contribution from the saddle-splay term, is particularly limiting for cholesterics. Interestingly, Nehring \& Saupe \cite{nehring1971elastic} showed that using a different approach to deriving the cholesteric free energy density, one deduces that $K_2+K_4\neq 0$. In this paper we will keep the contributions from the saddle-splay term and when we apply a one constant approximation, we will suppose that
\begin{equation}\label{0.5}
 K_1=K_2=K_3=K,\,\, K_4=0
\end{equation}
for both nematic and cholesteric liquid crystals.

\vspace{3 mm}

The applications of liquid crystals to computer display designs are well known and therefore any advance in our theoretical understanding of liquid crystals could potentially further their commercial viability. However, one of the main reasons for the abstract interest in cholesteric liquid crystals is that they can form a number of defect structures which are not fully understood analytically. The first of these are quasi-periodic cylindrical patterns (for example, see \cite[Fig. 5]{oswald2000static}) which are part of a large group of structures called cholesteric fingers, of which there are at least four distinct types \cite{oswald2000static,smalyukh2005electric}. These fingers can easily be created in a sample of liquid crystal between two parallel plates with homeotropic anchoring ($\vn = {\bf e}_3$), so long as the height of the cell is sufficiently small. It is generally thought that the fingers arise from the competition of two local energy minimisers; the constant state and a more helical state (an example of which we will find explicitly later in the paper). Near the isotropic transition temperature, some cholesteric liquid crystals form cubic lattices of cylinders called blue phases \cite{kikuchi2002polymer,wright1989crystalline} which possess interesting optical properties. They are called blue phases because the periodicity of the lattice is the same length as the wavelength of visible light, so it reflects a certain visible colour. The first such studied liquid crystals reflected blue light, hence their name. Even more recently, local structures called torons, which comprise a small volume of tightly twisted molecules, have been induced by a high power laser and shown to be stable \cite{smalyukh2010three}. These torons appear to be formed when the building block of the blue phase, the double twist cylinder, is twisted upon itself to create a torus. This structure necessarily results in two defects.

\vspace{3 mm}

Modern technological advances have made it possible to unambiguously reconstruct the director field within any of the above cholesteric structures \cite{smalyukh2005electric}. A experimental technique called fluoresence confocal polarizing microscopy (FCPM) is used to perform this reconstruction. As a result, the molecular makeup of all of these structures have been found either experimentally or through simulations. However there has been little analytical progress using them to find local or global energy minimisers of an entire system. The work which follows in this paper is aimed at being able to predict physically viable states of cholesteric liquid crystal problems.

\vspace{3 mm}

Although we will not be able to prove the stability of multi-dimensional minimisers in this paper directly, we will establish a great deal of groundwork by finding one-dimensional minimisers and investigating their stability. In Section \ref{section: one var mins} we study the Oseen-Frank free energy with the frustrated boundary conditions $\vn|_{z=0}={\bf e}_1$ and $\vn|_{z=L_3} = {\bf e}_3$. We find the unique global minimisers of the one dimensional problem for general elastic constants. From then on we consider the one constant Oseen-Frank energy
\begin{equation}\label{0.6}
 I(\vn) = \int_\Omega \left|\nabla \vn\right|^2 + 2t \vn \cdot \nabla \times \vn +t^2 \, dx.
\end{equation}
In Section \ref{section: long pitch} we prove the first of our two main stability results, Theorem \ref{2:theorem_global_min}: these one-dimensional states are in fact the global minimisers of \eqref{0.6} so long as $t$ is sufficiently small. This analytical prediction is consistent with the notion that cholesterics liquid crystals with suitably long pitch can be considered as a perturbation of nematic liquid crystals. We then investigate the relationship between the saddle-splay term, the function space and periodicity in the lateral directions with Proposition \ref{3: prop periodicity}. This result gives a subtle indication that there may be a function space modelling issue for cholesterics, this idea will be further explored in an upcoming paper \cite{bedford2014function}.

\vspace{3 mm}

In the final section we examine the same problem but with homeotropic boundary conditions on both the top and bottom faces of the cuboid. This problem is a little different in that for any $t\in \mathbb{R}$, $\vn = {\bf e}_3$ is always a solution of the Euler-Lagrange equation. Just as in the frustrated case we prove that $\vn = {\bf e}_3$ is a global minimiser of the problem provided that $t$ is sufficiently small. This is an interesting supplement to the local stability result \cite[Section 7]{bedford2014analysis} which demonstrated that $\vn = {\bf e}_3$ is a strong local minimiser if $t<\pi$ and not a local minimiser when $t>\pi$. All of these stability results give credence to the notion that cholesteric liquid crystals with long pitch can be understood as a perturbation of the more easily understood nematic liquid crystal systems.

%
%
%
%
%
%

\section{Notation and Preliminaries}\label{section: prelims}

Throughout this paper, unless stated otherwise, $\Omega \subset \mathbb{R}^3$ is the cuboid domain 
\begin{equation}\label{1.1}
 \Omega := (-L_1,L_1)\times (-L_2,L_2) \times (0,1).
\end{equation}
We note that we can assume, without loss of generality, that the domain has height 1 because scaling the domain simply corresponds to a linear scaling of the variable $t$. Then the problem that we are considering is the minimisation of 
\begin{equation}\label{1.2}
 I(\vn) := \int_\Omega w(\vn,\nabla \vn) \, dx
\end{equation}
over 
\begin{equation}\label{1.3}
 \mathcal{A} := \left\{ \left.\, \vn \in W^{1,2}\left( \Omega ,\mathbb{S}^2 \right) \, \right|\, \vn|_{z=0} = {\bf e}_1,\,\, \vn|_{z=1} = {\bf e}_3,\,\, \vn|_{x=-L_1} = \vn|_{x=L_1}, \,\, \vn|_{y=-L_2} = \vn|_{y=L_2} \, \right\},
\end{equation}
where 
\begin{equation}\label{1.4}
 w(\vn,\nabla \vn) := K_1\left( \nabla \cdot \vn \right)^2 +K_2 \left( \vn \cdot \nabla\times \vn +t \right)^2 +K_3 \left| \vn\times \nabla \times \vn \right|^2 +\left( K_2 + K_4 \right) \left( {\rm{tr}}(\nabla \vn^2)-(\nabla \cdot \vn)^2 \right).
\end{equation}
The Sobolev space $W^{1,2}\left( \Omega ,\mathbb{S}^2 \right)$ is simply the set of functions $\vn \in W^{1,2} \left( \Omega ,\mathbb{R}^3 \right)$ such that $|\vn|=1$ almost everywhere. In the final section we will briefly discuss local minimisers of this vectorial problem and therefore will need to utilise some definitions and results from that area of mathematics.

\begin{definition}\label{def: strong local min}
 A fucntion $\vn \in \mathcal{A}$ is a strong local minimiser of the functional $I$ if there exists an $\epsilon>0$ such that if $\vm \in \mathcal{A}$ and $||\vn-\vm||_{\infty}<\epsilon$ then $I(\vm)\geqslant I(\vn)$.
\end{definition}

\begin{definition}\label{def: weak local min}
 A fucntion $\vn \in \mathcal{A}$ is a weak local minimiser of the functional $I$ if there exists an $\epsilon>0$ such that if $\vm \in \mathcal{A}$ and $||\vn-\vm||_{1,\infty}<\epsilon$ then $I(\vm)\geqslant I(\vn)$.
\end{definition}

We define the set of test functions of our problem to be $\rm{Var}_\mathcal{A}$ where 
 \begin{equation}\label{1.5}
  {\rm{Var}}_\mathcal{A} := \left\{ \left.\, \vv \in C^\infty \left( {\Omega} , \mathbb{R}^3 \right)\cap W^{1,2}\left( \Omega,\mathbb{R}^3 \right)\, \right|\, \vv|_{z=0} = \vv|_{z=1} = 0,\,\, \vv|_{x=-L_1} = \vv|_{x=L_1},\,\, \vv|_{y=-L_2} = \vv|_{y=L_2} \, \right\}.
 \end{equation}

\begin{theorem}\cite{bedford2014analysis}\label{thm: weak local min}
 Suppose that $\vn \in \mathcal{A} \cap W^{1,\infty}\left(\Omega,\mathbb{R}^3 \right)$ is a weak local minimiser of $I$. Then for every $\vv \in \rm{Var}_\mathcal{A}$, 
 \begin{equation}\label{1.6}
   \left.\frac{d}{d\epsilon} I\left( \frac{\vn+\epsilon \vv}{|\vn+\epsilon\vv|} \right) \right|_{\epsilon=0} = 0\quad \text{and} \quad  \left.\frac{d^2}{d\epsilon^2} I\left( \frac{\vn+\epsilon \vv}{|\vn+\epsilon\vv|} \right) \right|_{\epsilon=0} \geqslant 0.
 \end{equation}

\end{theorem}

\begin{theorem}\cite{bedford2014analysis}\label{thm: strong local min}
 Suppose that $\vn \in \mathcal{A} \cap C^1 \left( \overline{\Omega},\mathbb{R}^3 \right)$. If there exists some $\delta>0$ such that
 \begin{equation}\label{1.7}
  \left.\frac{d}{d\epsilon} I\left( \frac{\vn+\epsilon \vv}{|\vn+\epsilon\vv|} \right) \right|_{\epsilon=0} = 0\quad \text{and} \quad  \left.\frac{d^2}{d\epsilon^2} I\left( \frac{\vn+\epsilon \vv}{|\vn+\epsilon\vv|} \right) \right|_{\epsilon=0} \geqslant \delta ||\vv-(\vv\cdot\vn)\vn||_{1,2}^2
 \end{equation}
 for every $\vv \in \rm{Var}_\mathcal{A}$, then $\vn$ is a strong local minimiser of $I$.

\end{theorem}

\section{One variable minimisers}\label{section: one var mins}

From here on we will be assuming that $t\geqslant 0$ and the reason we can do this is that the sign of $t$ simply determines the direction of the helix that the molecules prefer to form. Thus if $t<0$ we simply perform the change of variables ${\bf x}' = -{\bf x}$ to derive a situation where $t\geqslant 0$. In this section we will assume that ${\bf n}$ is a function of $z$ only, which will simplify the free energy significantly, and to avoid degeneracy we also assume that $K_1,K_2,K_3>0$. For maximum simplification we rewrite the director as a function of its Euler angles $\theta(z),\phi(z)$. By substituting 
\begin{equation}\label{2.39}
 \vn=\left( \begin{array}{c} \cos\phi\cos\theta \\\sin\phi\cos\theta\\\sin\theta \end{array} \right), 
\end{equation}
into \eqref{1.2}, we lose any contribution from the saddle-splay term since it vanishes for functions which depend on only one variable. Then we simplify to find
\begin{equation}\label{2.40}
\begin{split}
 I({\bf n})=4L_1L_2&\int^1\left(K_1\cos^2\theta +K_3\sin^2\theta\right)\theta'^2+\cos^2\theta\left[\left(K_2\cos^2\theta+K_3\sin^2\theta\right)\phi'^2-2tK_2\phi'\right] \\&+t^2K_2\,dz.
\end{split}
\end{equation}
Converting the problem into one involving the Euler angles is convenient from a computational standpoint as it deals with the unit vector constraint but there are regularity issues involved, especially at the poles where $\phi$ is undefined. Intuitively it seems sensible that $\theta$ and $\phi$ should have a comparable regularity to the director field $\vn$. However by simple calculations we can see that 
\begin{equation}\label{2.41}
 \int_{\Omega} |\nabla \vn|^2\,dx=\int_{\Omega}|\nabla \theta |^2+\cos^2\theta\,|\nabla \phi|^2\,dx<\infty \,\,\forall \,\,\vn\in W^{1,2}(\Omega,\mathbb{S}^2).
\end{equation}
This means that formally we might only expect $\nabla\theta,\cos\theta\,\nabla\phi \in L^2(\Omega,\mathbb{R}^3)$ so that $\theta \in W^{1,2}$ but we need not have that $\nabla\phi \in L^2$. The following proposition proves this regularity for $\theta$ and fortunately the precise form of the Lagrangian in \eqref{2.40} will then allow us to deal with the terms involving $\phi$ relatively simply. 

\begin{proposition}\label{ThetaRegularity}
 Let $d \in \mathbb{N}$ and suppose that $\Omega\subset\mathbb{R}^d$ is an open set. If $\vn \in W^{1,2}\left( \Omega,\mathbb{S}^2\right)$ then
 \begin{equation}\label{2.45.1}
  \theta \in W^{1,2}\left( \Omega , \left[ -\frac{\pi}{2} , \frac{\pi}{2} \right] \right)
 \end{equation}
 where $\theta$ is the Euler angle out of the $(x,y)$-plane associated with $\vn$.

\end{proposition}

\begin{proof}
 
 The first element of the proof is to show the following pointwise inequality for any $\vn_1,\vn_2 \in \mathbb{S}^2$:
 \begin{equation}\label{2.45.2}
  \left| \vn_1 -\vn_2 \right|^2 \geqslant C \left| \theta_1 -\theta_2 \right|^2.
 \end{equation}
 We take $\vn_1, \vn_2 \in \mathbb{S}^2$ then in terms of their Euler angles 
 \begin{equation}\label{2.45.3}
  \vn_1 = \left( \begin{array}{c} \cos \theta_1 \cos \phi_1 \\ \cos \theta_1 \sin \phi_1 \\ \sin \theta_1 \end{array} \right) \quad
  \vn_2 = \left( \begin{array}{c} \cos \theta_2 \cos \phi_2 \\ \cos \theta_2 \sin \phi_2 \\ \sin \theta_2 \end{array} \right),
 \end{equation}
 where $\theta_1,\theta_2 \in \left[ -\frac{\pi}{2},\frac{\pi}{2} \right]$ and $\phi_1,\phi_2 \in \left[0,2\pi\right)$.  We substitute \eqref{2.45.3} into the left hand side of \eqref{2.45.2} to find
 \begin{equation}\label{2.45.4}
  \begin{split}
   |\vn_1-\vn_2|^2 & = 2-2\vn_1 \cdot \vn_2 \\
   &= 2-2\sin\theta_1\sin\theta_2-2\cos\theta_1\cos\theta_2 \cos(\phi_1-\phi_2) \\
   & \geqslant 2-2\sin\theta_1\sin\theta_2-2\cos\theta_1\cos\theta_2 \\
   & = 2\left( 1- \cos(\theta_1-\theta_2)\right).
  \end{split}
 \end{equation}
 A simple exercise in L'H\^{o}ptial's rule gives us that 
 \begin{equation}\label{2.45.5} 
  \lim_{x \rightarrow 0} \frac{1-\cos(x)}{x^2} = \frac{1}{2}.
 \end{equation}
 Therefore
 \begin{equation}\label{2.45.6}
  \inf_{x\in [-\pi,\pi]} \frac{1-\cos(x)}{x^2} :=C > 0.
 \end{equation}
 Then by combining \eqref{2.45.6} and \eqref{2.45.4} we see that $|\vn_1-\vn_2|^2 \geqslant C |\theta_1 -\theta_2|^2$. We can now conclude the assertion by using the difference quotient characterisation of $W^{1,2}$ \cite[Thm 10.55]{leoni2009first}. We take our $\vn \in W^{1,2}\left( \Omega,\mathbb{S}^2 \right)$ and let $x_0 \in \Omega$ be a point where $\vn$ is well defined, then 
 \begin{equation}\label{2.45.8}
  \int_\Omega |\vn(x)-\vn(x_0)|^2\,dx \geqslant C\int_\Omega |\theta(x)-\theta(x_0)|^2\,dx <\infty,
 \end{equation}
 so that $\theta \in L^2\left( \Omega,\mathbb{R}\right)$. The form of the inequality \eqref{2.45.2} allows us to go further and say that 
 \begin{equation}\label{2.45.9}
  ||\vn(\cdot + h{\bf e}_j)-\vn(\cdot) ||_{2,\Omega'} \geqslant C^{\frac{1}{2}} ||\theta(\cdot + h{\bf e}_j)-\theta(\cdot) ||_{2,\Omega'}
 \end{equation}
for every $\Omega' \subset\subset \Omega$, $h\in \mathbb{R}$ and $j\in \left\{ 1,\dots,n\right\}$. Therefore as $\vn$ satisfies the difference quotient characterisation, so must $\theta$, which gives us the result.

\end{proof}

Now that we know the precise regularity of $\theta$ we can return to \eqref{2.40}. Clearly with regards to minimisation, the final constant term in \eqref{2.40} is unimportant as its energy contribution is the same for all states. We apply a very direct approach to minimising the contributions from $\phi$. If we look simply at the central term involving $\phi$ we can minimise it formally. To do this we note that
\begin{equation}\label{2.42}
 f(x):=ax^2-bx \geqslant f\left( \frac{b}{2a} \right) = -\frac{b^2}{4a}.
\end{equation}
Thus if we simply set 
\begin{equation}\label{2.44}
 \phi'=\frac{K_2t}{\left(K_2\cos^2\theta+K_3\sin^2\theta\right)}\quad \phi(0)=0,
\end{equation}
then we have minimised the contribution with regards to $\phi'$ and have the following inequality
\begin{equation}\label{2.45}
\begin{split} I({\bf n})&\geqslant 4L_1L_2\int_0^1\left(K_1\cos^2\theta +K_3\sin^2\theta\right)\theta'^2-\frac{K_2^2t^2\cos^2\theta}{\left(K_2\cos^2\theta+K_3\sin^2\theta\right)}+K_2t^2\,dz\\
 &=:4L_1L_2\int_0^1 f(\theta)\theta'^2-g(\theta)+K_2t^2\,dz.
\end{split}
\end{equation}
This reduces the problem to a one-dimensional functional of one variable which is easier to investigate. In order to find its minimum we first need to prove a preliminary lemma whereby we will be able show that it is not energetically favourable for $\theta$ to achieve either of its bounds, $\pm \frac{\pi}{2}$, except at $z=1$. For the lemma we introduce the following notation:
\begin{equation}\label{2.48}
 \begin{split}
&F_{\alpha}(v):=\int_0^{\alpha} f(v)v'^2-g(v)+K_2t^2\,dz,\\
&\mathcal{A}_{\alpha}:=\left\{\,\left.v\in W^{1,2}(0,\alpha)\,\right|\,v(0)=0,\,\,v(\alpha)=\frac{\pi}{2}\,\right\}.
 \end{split}
\end{equation}

\begin{lemma}\label{2:lemma_poles}
 If $0<\alpha<\beta<\infty$ then 
\begin{equation}\label{2.49}
 \inf_{v\in \mathcal{A}_{\alpha}} F_{\alpha}(v) > \inf_{w\in \mathcal{A}_{\beta}} F_{\beta}(w).
\end{equation}
\end{lemma}

\begin{proof}

We know that $\inf_{v\in\mathbb{R}} f(v)=\min\left\{K_1,K_3\right\}>0$ so that the Lagrangian is convex with respect to the gradient. Therefore the 1-dimensional theory of the calculus of variations tells us that for each $\alpha$, the problem of minimising $F_{\alpha}$ over $\mathcal{A}_{\alpha}$ has a smooth minimiser \cite[Ch.4.2]{dacorogna2008direct} which satisfies the following Euler-Lagrange equation
\begin{equation}\label{2.50}
 (2f(v)v')'=-g'(v).
\end{equation}
This equation has the first integral
\begin{equation}\label{2.51}
 f(v)v'^2+g(v)=C_{\alpha},
\end{equation}
where $C_{\alpha}$ is a constant. For the consistency of this equation $C_{\alpha}\geqslant K_2t^2$ because
\begin{equation}\label{2.52}
f(v(z))>0,\quad g(v(z))\leqslant K_2t^2,\quad \text{and}\quad g(v(0))=K_2t^2 .
\end{equation}
Importantly we need to eliminate the possibility that $C_{\alpha}=K_2t^2$ so that we can deduce $v'$ has a fixed sign. If $C_{\alpha}=K_2t^2$, then we would be looking to solve 
\begin{equation}\label{2.53}
 v'^2= \frac{K_2 t^2 - g(v)}{f(v)}\quad v(0)=0.
\end{equation}
We know that the minimiser we are searching for is smooth, hence a simple argument looking at intervals where $v'$ has fixed sign shows that the unique solution to \eqref{2.53} is $v=0$ and this clearly fails to match our boundary conditions. Hence $C_{\alpha}>K_2t^2$ and $v'$ has a fixed sign, so taking the square-root of \eqref{2.53} 
gives
\begin{equation}\label{2.54}
 v'=\frac{\left(C_{\alpha}-\frac{K_2^2t^2\cos^2v}{\left(K_2\cos^2v+K_3\sin^2v\right)}\right)^{\frac{1}{2}}}{\left(K_1\cos^2v +K_3\sin^2v\right)^\frac{1}{2}}:=k(z,v)
\quad v(0)=0.
\end{equation}
We can use the Picard-Lindel\"{o}f Theorem here to deduce existence and uniqueness of this ODE \cite[p.8]{hartman1964ordinary}. It is applicable since the function 
$k(z,v)$ is smooth in its arguments, hence Lipschitz on any bounded set. This means that we have a unique solution of this ODE, $v_{\alpha}$, given implicitly by
\begin{equation}\label{2.55}
  \int_0^{v_{\alpha}(z)} \frac{\left(K_1\cos^2u+K_3\sin^2u\right)^{\frac{1}{2}}}{\left(C_{\alpha}-\frac{K_2^2t^2\cos^2u}{\left(K_2\cos^2u+K_3\sin^2u\right)}\right)^{\frac{1}{2}}}\,du=z.
\end{equation}
We finalise the definition of $v_{\alpha}$ by applying the top boundary condition of $v_{\alpha}(\alpha)=\frac{\pi}{2}$. This is possible because the functional
\begin{equation}\label{2.56}
 \eta(C):=\int_0^{\frac{\pi}{2}} \frac{\left(K_1\cos^2u+K_3\sin^2u\right)^{\frac{1}{2}}}{\left(C-\frac{K_2^2t^2\cos^2u}{\left(K_2\cos^2u+K_3\sin^2u\right)}\right)^{\frac{1}{2}}}\,du,
\end{equation}
is a strictly monotone decreasing function satisfying 
\begin{equation}\label{2.57}
 \eta(K_2t^2)=\infty\quad \text{and}\quad \eta(C)\rightarrow0\quad \text{as}\quad C\rightarrow \infty.
\end{equation}
This means that there is a unique solution to the equation $\eta(C_{\alpha})=\alpha$. Therefore we have established that for all $\alpha>0$ we have found the unique smooth $v_{\alpha}$ satisfying
\begin{equation}\label{2.58}
 F_{\alpha}(v_{\alpha})=\inf_{v\in \mathcal{A}_{\alpha}} F_{\alpha}(v).
\end{equation}
Now we note that for $\beta>\alpha>0$
\begin{equation}\label{2.59}
 v(z):=\left\{ \begin{array}{lcl} 0&\text{if}&z\in(0,\beta-\alpha)\\&&\\v_{\alpha}(z-(\beta-\alpha))&\text{if}&z\in(\beta-\alpha,\beta)\end{array}\right. \in \mathcal{A}_{\beta},
\end{equation}
with the property that 
\begin{equation}\label{2.60}
 F_{\beta}(v)=F_{\alpha}(v_{\alpha})=\inf_{w\in \mathcal{A}_{\alpha}}F_{\alpha}(w).
\end{equation}
An immediate consequence of \eqref{2.60} is that 
\begin{equation}\label{2.61}
F_{\beta}(v_{\beta})= \inf_{u\in \mathcal{A}_{\beta}}F_{\beta}(u)\leqslant F_{\beta}(v)=\inf_{w\in \mathcal{A}_{\alpha}}F_{\alpha}(w).
\end{equation}
To conclude that this inequality is strict we need to show that $v\neq v_{\beta}$. However we know this to be true because $C_{\beta}>K_2t^2$ tells us that $v_{\beta}'(z)\neq 0$ for any $z\in [0,1]$ whereas $v'(z)=0$ for $z\in(0,\beta-\alpha)$ in \eqref{2.59}. Hence the inequality in \eqref{2.61} is strict.

\end{proof}

We can now prove the main result of this section by finding the one variable global minimisers.

\begin{theorem}\label{2:1varmin}

The function $\theta^*=\theta^*(z)$ given implicitly by the equation 
\begin{equation}\label{2.46}
 \int_0^{\theta^*} \frac{\left(K_1\cos^2u+K_3\sin^2u\right)^{\frac{1}{2}}}{\left(C-\frac{K_2^2t^2\cos^2u}{\left(K_2\cos^2u+K_3\sin^2u\right)}\right)^{\frac{1}{2}}}\,du=z \quad \theta^*\left(1 \right) = \frac{\pi}{2},
\end{equation}
with $\phi^*$ as defined by \eqref{2.44}, determines a director field ${\bf n}^*$ which is the global minimum of $I$ over admissible functions which depend only on $z$. Moreover the constant $C$ in \eqref{2.46} is determined by the unique solution to
\begin{equation} \label{2.47}
 \int_0^{\frac{\pi}{2}} \frac{\left(K_1\cos^2u+K_3\sin^2u\right)^{\frac{1}{2}}}{\left(C-\frac{K_2^2t^2\cos^2u}{\left(K_2\cos^2u+K_3\sin^2u\right)}\right)^{\frac{1}{2}}}\,du=1.
\end{equation}
\end{theorem}

\begin{proof}

We take an arbitrary $\vn\in\mathcal{A}$ which depends only on $z$. We know that 
\begin{equation}\label{2.62}
\left. \begin{array}{ccl} I({\bf n})&\geqslant& 4L_1L_2\int_0^1\left(K_1\cos^2\theta +K_3\sin^2\theta\right)\theta'^2-\frac{K_2^2t^2\cos^2\theta}{\left(K_2\cos^2\theta+K_3\sin^2\theta\right)}+K_2t^2\,dz\\&&\\&=&4L_1L_2\int_0^1 f(\theta)\theta'^2-g(\theta)+K_2t^2\,dz.\end{array}\right. 
\end{equation}
We also know that $\theta\in W^{1,2}(0,1)\subset\subset C[0,1]$, so $\theta$ is a continuous function. Lemma \ref{2:lemma_poles} tells us that if $\theta(\alpha)=\pm \frac{\pi}{2}$ for some $0<\alpha<1$ then we cannot have the minimum energy of \eqref{2.62}(ii). In short
\begin{equation}\label{2.63}
 \int_0^1 f(\theta)\theta'^2-g(\theta)+K_2t^2\,dz\geqslant \int_0^1 f(\theta^*){\theta^*}'^2-g(\theta^*)+K_2t^2\,dz.
\end{equation}
where $\theta^*=v_1$ in the notation of Lemma \ref{2:lemma_poles}. Additionally the uniqueness of the minimiser from Lemma \ref{2:lemma_poles} for each positive $\alpha$ means that we have equality in \eqref{2.63} if and only if $\theta=\theta^*$. For completeness we should note that given a smooth $\theta^*$, the equation involving $\phi$, \eqref{2.44}, defines a unique smooth function $\phi^*(z)$ by Picard-Lindel\"{o}f \cite[p.8]{hartman1964ordinary}. So now we can say that 
\begin{equation}\label{2.64}
 I(\vn)\geqslant 4L_1L_2\int_0^1 f(\theta^*){\theta^*}'^2-g(\theta^*)+K_2t^2\,dz=I(\vn^*),
\end{equation}
with equality if and only if $\theta=\theta^*$ and also 
\begin{equation}\label{2.65}
\phi'=\frac{K_2t}{\left(K_2\cos^2\theta+K_3\sin^2\theta\right)}\quad\text{and}\quad \phi(0)=0 .
\end{equation}
Hence $I(\vn)\geqslant I(\vn^*)$ with equality if and only if $\theta=\theta^*$ and $\phi=\phi^*$, or equivalently $\vn=\vn^*$. 

\end{proof}

\begin{remark}

Theorem \ref{2:1varmin} concludes the discussion on the minimisers of functions of $z$ alone. However not having a closed form for these minimisers is a significant problem since we want to find explicit minimisers in order to compare with experimental observations. In order to achieve this we need to assign values to the elastic constants, we will use the one constant approximation because it makes the analysis most tractable.

\end{remark}

%
%
%
%
%
%

\section{Long cholesteric pitch}\label{section: long pitch}

When we apply the one constant approximation to the general Oseen-Frank free energy, the cholesteric free energy reduces to
\begin{equation}\label{2.91}
I({\bf n})=K\int_{\Omega} |\nabla {\bf n}|^2+2t{\bf n}\cdot\nabla\times{\bf n}+t^2\,dx,
\end{equation}
with the same set of admissible functions 
\begin{equation}\label{2.92}
 \mathcal{A}:=\left\{\,\left.\vn\in W^{1,2}(\Omega,\mathbb{S}^2)\,\right|\,\vn|_{z=0}= {\bf e}_1,\,\,\vn|_{z=1}= {\bf e}_3,\,\,\vn|_{x=-L_1}=\vn|_{x=L_1} ,\,\,\vn|_{y=-L_2}= \vn|_{y=L_2}
\,\right\}.
\end{equation}
For ease we will drop the now unimportant elastic constant $K$ from our functional. We already know the form of the minimiser amongst functions of one variable of this functional from Theorem \ref{2:1varmin}. Hence we begin this section with a simple corollary which is just a specific case of Theorem \ref{2:1varmin}.

\begin{corollary}\label{2: cholesteric one variable min}

The function, $\theta(z)$, given implicitly by the equation 
\begin{equation}\label{2.93}
 \int_0^{\theta(z)} \frac{1}{\left(D-t^2\cos^2u\right)^{\frac{1}{2}}}\,du=z,
\end{equation}
where the constant $D$ is defined by
\begin{equation} \label{2.94}
 \int_0^{\frac{\pi}{2}} \frac{1}{\left(D-t^2\cos^2u\right)^{\frac{1}{2}}}\,du=1,
\end{equation}
and
\begin{equation}\label{2.95}
 \phi=tz,
\end{equation}
determines a director field ${\bf n}^*$ which is the unique global minimum of \eqref{2.91} over $\vn \in \mathcal{A}$ which depend only on $z$.
\end{corollary}

\begin{proof}

By substituting $K_1=K_2=K_3=K$ into \eqref{2.46} and \eqref{2.47}, and by setting 
\begin{equation}\label{2.96}
 D=\frac{C}{K},
\end{equation}
we immediately obtain the result

\end{proof}

Before we proceed further, to better understand our one variable minimiser it is constructive to compute the explicit solutions of \eqref{2.93} and \eqref{2.94} for various values of $t$. For all of the numerical calculations the standard 
ODE solver in Matlab was used, a fourth order Runge-Kutta method. The following four figures: Figures \ref{fig:s=5}, \ref{fig:s=10}, \ref{fig:s=20}, and \ref{fig:s=40}, each show a plot of 
$\theta$ against the height of the cell, each for a different value of the cholesteric twist strength, $t$.

\vspace{3 mm}

\begin{figure}[ht!]
\centering
\subfloat[$t=2.5$]{\label{fig:s=5}\includegraphics[scale=0.21]{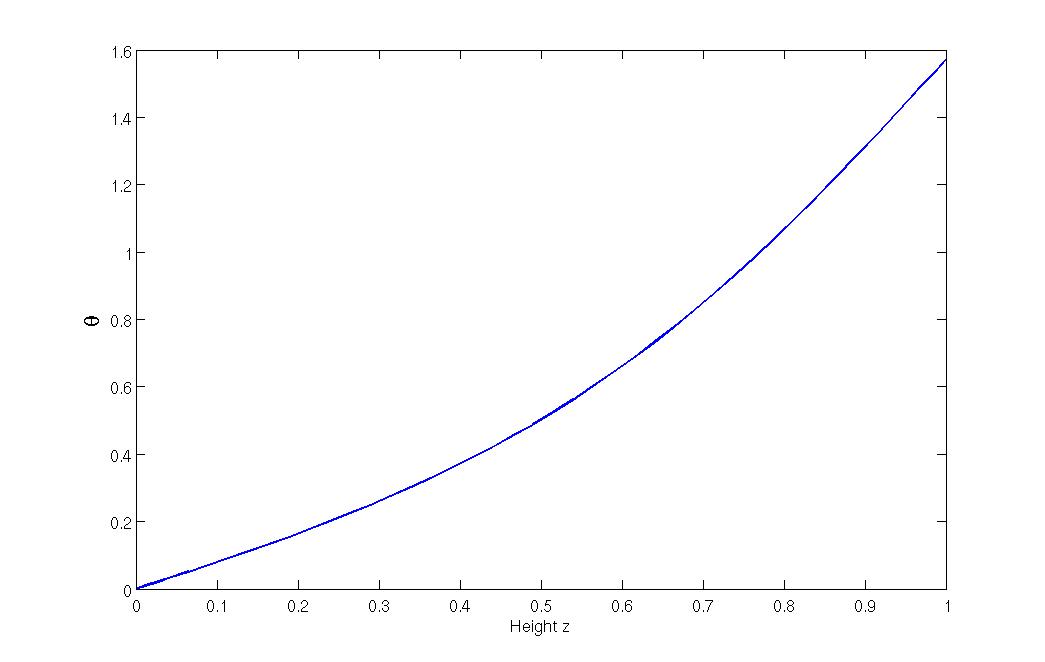}}
\subfloat[$t=5$]{\label{fig:s=10}\includegraphics[scale=0.21]{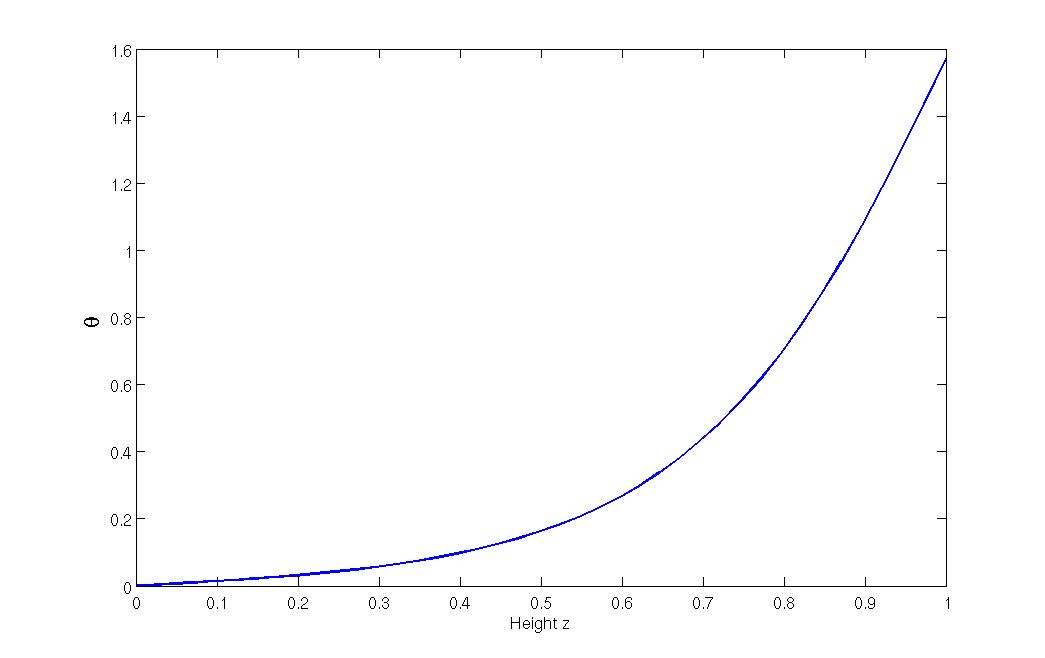}}\\
\subfloat[$t=10$]{\label{fig:s=20}\includegraphics[scale=0.21]{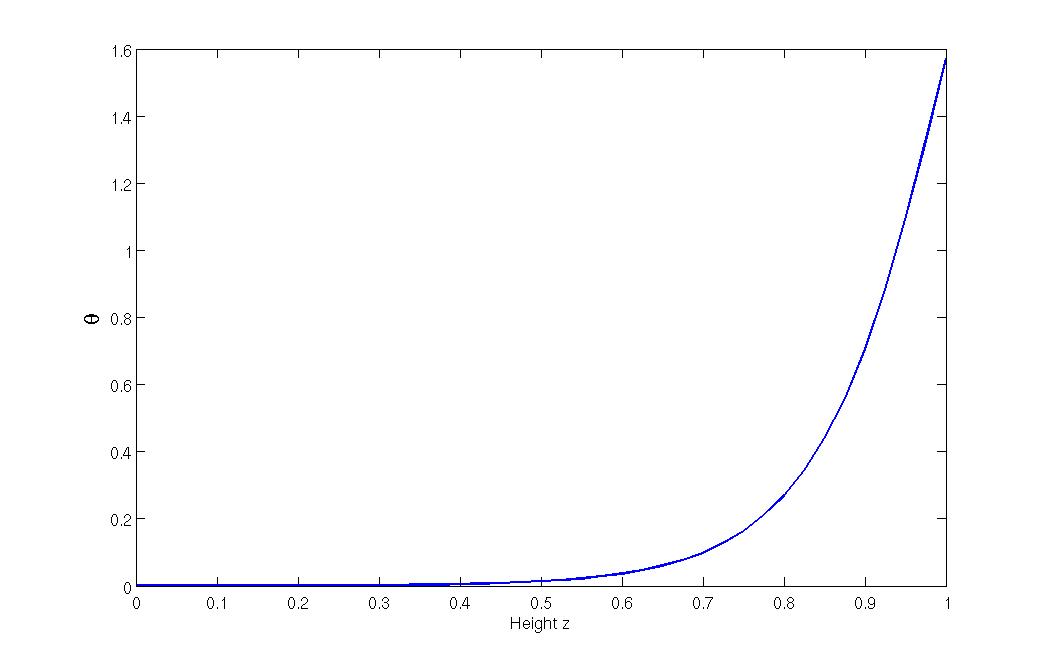}}
\subfloat[$t=20$]{\label{fig:s=40}\includegraphics[scale=0.21]{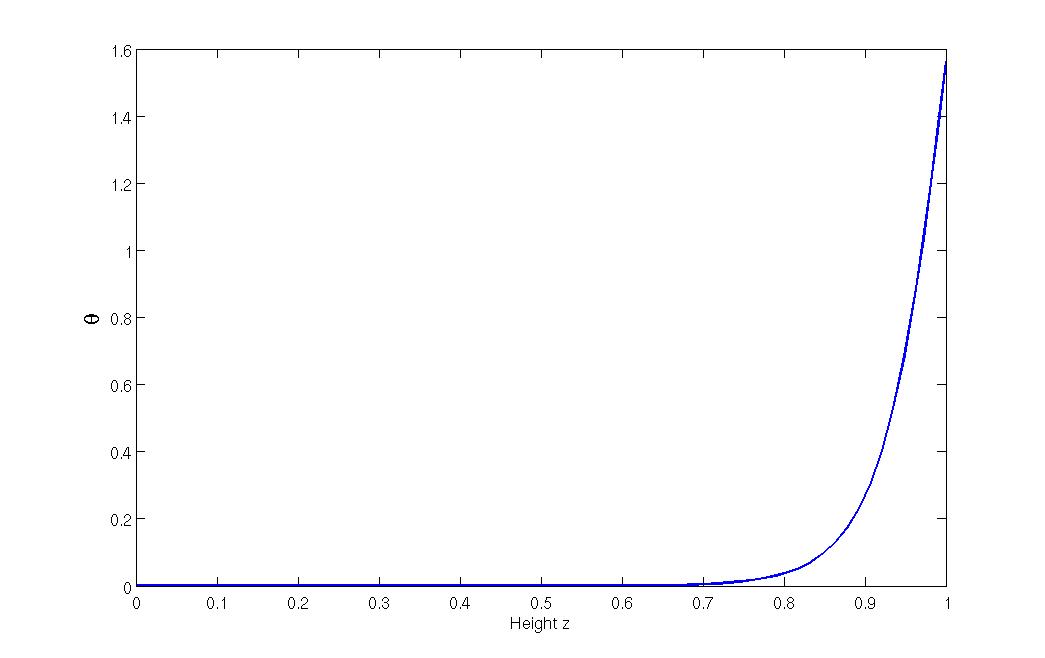}}
\caption{$\theta$ plots for various values of $t$}
\label{2:fig_thetas}
\end{figure}

The trend is clear to see. Small values of $t$ correspond to a slight bending of the 
case $t=0\,\,(\Rightarrow \theta = \frac{\pi z}{2})$. Larger values of $t$ have much more curvature towards the top of the cell. The reason for this can be seen by looking at \eqref{2.91} 
and substituting $\phi'=t$. When we do so, we find that the energy, purely in terms of $\theta$, is given by 
\begin{equation}\label{2.96b}
 \int_0^1 \theta'^2+t^2\sin^2\theta\,dz.
\end{equation}
Hence it is no surprise that as $t$ increases the second term becomes more dominant, meaning that $\theta$ remains small so as to minimise $\sin^2 \theta$.

\vspace{3 mm}

We return now to the problem set out in \eqref{2.91} and \eqref{2.92}. We have found the global minimiser of this problem over functions of $z$ alone. The very natural question from that point is to ask whether these functions, obtained 
numerically in Figure \ref{2:fig_thetas}, are the global minimisers of the entire problem. The question has been partly answered using Theorem \ref{2:theorem_global_min} below. 

\vspace{3 mm}

Firstly and most importantly, in order to check whether our candidate minimiser is a global minimiser we need to check that it satisfies the Euler-Lagrange equation. Hence we need to derive the equation, which we do in the following lemma.

\begin{lemma}\label{2:lemma_chol_euler_lagrange}

If we have ${\bf n}\in\mathcal{A}\cap W^{1,\infty}\left(\Omega,\mathbb{S}^2 \right)$ which is a weak local minimiser of \eqref{2.91}. Then it satisfies the following Euler-Lagrange equation (in the sense of distributions)
\begin{equation}\label{2.100}
 \Delta {\bf n}-2t\nabla\times{\bf n}+\left(|\nabla{\bf n}|^2+2t{\bf n}\cdot\nabla\times{\bf n}\right){\bf n}=0,
\end{equation}
together with the natural boundary conditions
\begin{equation}\label{2.100.1}
 \left(\nabla \vn\right) \bm\nu |_{x=-L_1} = \left(\nabla \vn\right) \bm \nu |_{x=L_1},\,\, \left(\nabla \vn\right) \bm\nu |_{y=-L_2} = \left(\nabla \vn\right) \bm \nu |_{y=L_2}.
\end{equation}

\end{lemma}

\begin{proof}

We take some $\vphi \in \rm{Var}_\mathcal{A}$ and note that 
\begin{equation}\label{2.100.2}
 \frac{\vn+\epsilon \vphi}{|\vn+\epsilon \vphi|} \rightarrow \vn \,\, \text{in}\,\, W^{1,\infty}\,\, \text{as}\,\, \epsilon \rightarrow 0.
\end{equation}
Therefore, because $\vn$ is a weak local minimiser, it must be the case that the first variation vanishes
\begin{equation}\label{2.100.3}
 \begin{split}
  0 & = \left. \frac{d}{d\epsilon} I \left( \frac{\vn+\epsilon\vphi}{|\vn+\epsilon \vphi|} \right) \right|_{\epsilon=0} \\
   & = 2 \int_\Omega \nabla \vphi : \nabla \vn - (\vn \cdot \vphi)\left( |\nabla \vn|^2 + 2t \vn \cdot \nabla \times \vn \right) +t\left( \vphi \cdot \nabla \times \vn + \vn \cdot \nabla \times \vphi\right)\, dx.
 \end{split}
\end{equation}
By using integration by parts, together with the boundary conditions for both $\phi$ and $\vn$, we can deduce that $\int_\Omega \vphi \cdot \nabla \times \vn \,dx = \int_\Omega \vn \cdot \nabla \times \vphi\,dx$. Hence
\begin{equation}\label{2.100.4}
 0 = \int_\Omega -\vphi \cdot \left( \Delta \vn -2t \nabla \times \vn +\vn \left( |\nabla \vn|^2 +2t\vn\cdot \nabla \times \vn\right) \right)\, dx + \int_{\partial \Omega} \vphi \cdot \left[ (\nabla \vn)\bm \nu \right] \, dS,
\end{equation}
for any $\vphi \in \rm{Var}_\mathcal{A}$. This proves equations \eqref{2.100} and \eqref{2.100.1} and completes the proof.

\end{proof}

\begin{theorem}\label{2:theorem_global_min}

There exists some $t^*>0$ such that whenever $t\leqslant t^*$, the function defined by \eqref{2.93}, \eqref{2.94}, and \eqref{2.95} is the 
unique global minimiser of 
\begin{equation}\label{2.101}
 I({\bf n})=\int_{\Omega} |\nabla {\bf n}|^2+2t{\bf n}\cdot\nabla\times{\bf n}+t^2\,dx,
\end{equation}
over \emph{all} admissible functions ${\bf n}\in\mathcal{A}$.

\end{theorem}

\begin{proof}

We have our candidate minimiser of this functional ${\bf n}^*$ which is given by the equation 
\begin{equation} \label{2.102}
 {\bf n}^*=\left( \begin{array}{c} \cos \theta  \cos(tz) \\ \cos\theta\sin(tz) \\ \sin\theta \end{array} \right) ,
\end{equation}
where $\theta$ is a solution of
\begin{equation}\label{2.103}
 \theta''=t^2\cos\theta\sin\theta,\quad\theta(0)=0,\quad\theta(1)=\frac{\pi}{2}.
\end{equation}
However we can find the first integral of \eqref{2.103} by multiplying by $\theta'$ and integrating to deduce that (for $\delta_{t}>0$)
\begin{equation}\label{2.104}
 \theta '^2 +t^2\cos^2\,\theta \,=\delta_{t} +t^2,\quad \theta(0)=0,\quad \theta(1)=\frac{\pi}{2} .
\end{equation}
The proof that this is indeed the global minimum has two main calculations.

\begin{itemize}
 \item Firstly we will show the following relation for an arbitrary ${\bf n}\in\mathcal{A}$
\begin{equation} \label{2.105}
H({\bf n}-{\bf n}^*) =I({\bf n})-I({\bf n}^*),
\end{equation}
where the functional $H$ is given by 
\begin{equation} \label{2.106} 
H({\bf v}) := \int_{\Omega} |\nabla {\bf v}|^2 +2t{\bf v}\cdot \nabla \times {\bf v} -\left(|\nabla {\bf n}^*|^2+2t{\bf n}^* \cdot \nabla \times {\bf n}^*\right)|{\bf v}|^2\,dx.
\end{equation}
$H$ takes functions from $\mathcal{B}$ which is the set given by 
\begin{equation} \label{2.107}
 \mathcal{B} :=\left\{\, \left.{\bf v}\in W^{1,2}\left(\Omega,\mathbb{R}^3\right) \,\right|\,{\bf v}|_{z=0}=0,\,\,{\bf v}|_{z=1}=0,\,\,{\bf v}|_{x=-L_1}={\bf v}|_{x=L_1},\,\, {\bf v}|_{y=-L_2}={\bf v}|_{y=L_2}\right\}.
\end{equation}
 \item Secondly we will show that if $t$ is small enough then for all $\vv \in \mathcal{B}$
\begin{equation}\label{2.108}
 H({\bf v})\geqslant \gamma_{t}\int_{\Omega} |\nabla {\bf v}|^2\,dx,
\end{equation}
for some $\gamma_{t}>0$.
\end{itemize}

\subsubsection*{$I(\vn)-I(\vn^*)$ Calculation}

Take an arbitrary ${\bf n} \in \mathcal{A}$ then 
\begin{equation} \label{2.109}
 \int_{\Omega} |\nabla {\bf n} -\nabla {\bf n}^*|^2 \,dx = \int_{\Omega} | \nabla {\bf n}|^2 +| \nabla {\bf n}^*|^2 -2\nabla {\bf n} : \nabla {\bf n}^* \,dx.
\end{equation}
Now we can use the Euler-Lagrange equation to rewrite the final term in the relation above. The strong form of the Euler-Lagrange equation which is satisfied by our function $\vn^*$ is 
\begin{equation}\label{2.109.1}
\Delta {\bf n}^*-2t\nabla\times{\bf n}^*+\left(|\nabla{\bf n}^*|^2+2t{\bf n}^*\cdot\nabla\times{\bf n}^*\right){\bf n}^*=0.
\end{equation}
We multiply this equation by $\vn$ and then we can use integration by parts to rewrite the equation as follows
\begin{equation}\label{2.109.2}
\begin{split}
 0&= \int_\Omega \vn \cdot \Delta \vn^* -2t\vn \cdot \nabla \times \vn^* + \lambda(\vn\cdot\vn^*)\,dx \\
  & = \int_\Omega \vn_i\vn^*_{i,33} - 2t\vn \cdot \nabla \times \vn^* + \lambda (\vn \cdot \vn^*)\,dx \\
 & = \int_\Omega -\vn_{i,3}\vn^*_{i,3} - 2t\vn \cdot \nabla \times \vn^* + \lambda (\vn \cdot \vn^*)\,dx \\
 & = \int_\Omega -\nabla \vn : \nabla \vn^* - 2t \vn \cdot \nabla \times \vn^* + \lambda (\vn \cdot \vn^*)\,dx,
\end{split}
\end{equation}
where $\lambda := |\nabla \vn^*|^2 + 2t\vn^* \cdot \nabla \times \vn^*$. We know that the integration by parts holds in the above calculation due to the definitions of $\vn$ and $\vn^*$. 
\begin{equation} \label{2.112}
 \vn_i\vn_{i,3}^*(0) = \vn^*_{1,3}(0)=\left. -t \sin\left(tz\right)\cos\,\theta\,-\theta ' \cos \left(tz\right)\sin\,\theta \, \right|_{z=0} =0,
\end{equation}
and similarly
\begin{equation} \label{2.113}
 \vn_i\vn_{i,3}^*(1) = \vn_{3,3}^*(1) =\left. \theta ' \cos \, \theta \, \right|_{z=1} =0.
\end{equation}
Thus \eqref{2.109} can be written as 
\begin{equation} \label{2.114}
  \int_{\Omega} |\nabla {\bf n} -\nabla {\bf n}^*|^2 \,dx = \int_{\Omega} | \nabla {\bf n}|^2 +| \nabla {\bf n}^*|^2 +4t\vn \cdot \nabla \times \vn^* -2\lambda (\vn \cdot \vn^*) \,dx.
\end{equation}
Using the facts that $2(1-{\bf n}\cdot {\bf n}^*)=|{\bf n}-{\bf n}^*|^2$ and that $I({\bf n}^*)=\int_{\Omega} \lambda+t^2 \,dx$, we deduce the following calculations:
\begin{equation} \label{2.116}
 \begin{split}\int_{\Omega} |\nabla {\bf n} -\nabla {\bf n}^*|^2 \,dx = & \int_{\Omega} | \nabla {\bf n}|^2 +| \nabla {\bf n}^*|^2 +4t{\bf n} \cdot \nabla \times {\bf n}^*
 -2\lambda {\bf n}\cdot {\bf n}^* \,dx \\
 = & \int_{\Omega} | \nabla{\bf n}|^2 +| \nabla {\bf n}^*|^2 +4t{\bf n} \cdot \nabla \times {\bf n}^* +2\lambda(1- {\bf n}\cdot {\bf n}^*) -2\lambda \,dx \\
 = & \int_{\Omega} | \nabla {\bf n}|^2 +| \nabla {\bf n}^*|^2 +4t{\bf n} \cdot \nabla \times {\bf n}^* +\lambda |{\bf n}-{\bf n}^*|^2 +2t^2\,dx -2I({\bf n}^*) \\
 = & I({\bf n})+I({\bf n}^*) +2t\int_{\Omega} 2({\bf n}\cdot \nabla \times {\bf n}^*) -({\bf n}\cdot \nabla \times {\bf n}) - ({\bf n}^* \cdot \nabla \times {\bf n}^*)\\  &+ \lambda |{\bf n}-{\bf n}^*|^2 \,dx -2I({\bf n}^*) \\
 = & I({\bf n})-I({\bf n}^*) +2t\int_{\Omega} 2({\bf n}\cdot \nabla \times {\bf n}^*) -({\bf n}\cdot \nabla \times {\bf n}) - ({\bf n}^* \cdot \nabla \times {\bf n}^*) \,dx \\  &+ \int_{\Omega} \lambda |{\bf n}-{\bf n}^*|^2 \,dx .
\end{split}
\end{equation}
Looking at the central integral in the equation above, we can rearrange the expression to get
\begin{equation} \label{2.117}
\begin{split} 
 &\int_{\Omega} 2({\bf n}\cdot \nabla \times {\bf n}^*) -({\bf n}\cdot \nabla \times {\bf n}) - ({\bf n}^* \cdot \nabla \times {\bf n}^*) \,dx \\ 
        = & \int_{\Omega} \left[ ({\bf n}-{\bf n}^*)\cdot \nabla \times ({\bf n}^*-{\bf n})\right] +({\bf n}\cdot \nabla \times {\bf n}^*)-({\bf n}^*\cdot \nabla \times {\bf n}) \,dx .
\end{split}
\end{equation}
Then using integration by parts we find that
\begin{equation} \label{2.118}
 \int_{\Omega} {\bf n} \cdot \nabla \times {\bf n}^* \,dx=\int_{\Omega} \vn_i \epsilon_{ijk}\vn_{k,j}^* = \int_{\partial \Omega} \vn_i \vn_k^* \epsilon_{ijk} \bm\nu^j \,dS - 
\int_{\Omega} \vn_k^*\epsilon_{ijk} \vn_{i,j}\,dx =
\int_{\Omega} {\bf n}^* \cdot \nabla \times{\bf n} \,dx.
\end{equation}
The surface integral is easily seen to be zero by applying the boundary conditions from the set $\mathcal{A}$. We are now in a position to conclude the calculation by combining the 
final line of \eqref{2.116}, \eqref{2.117} and \eqref{2.118}:
\begin{equation} \label{2.119}
 \int_{\Omega} |\nabla {\bf n}-\nabla {\bf n}^*|^2 +2t\left[({\bf n}-{\bf n}^*)\cdot \nabla \times ({\bf n}-{\bf n}^*)\right] -\lambda|{\bf n}-{\bf n}^*|^2\,dx =
 H({\bf n}-{\bf n}^*) = I({\bf n})-I({\bf n}^*) .
\end{equation}

\subsubsection*{Integral estimate}

Now we look to find a lower bound for the integral functional $H$ and we begin by noting the following relation obtained just by simple algebra (using \eqref{2.102} and \eqref{2.104})
\begin{equation}\label{2.120}
 |\nabla {\bf n}^*|^2+2t{\bf n}^* \cdot \nabla \times {\bf n}^*=\delta_{t} -t^2+2t^2\sin^2 \theta.
\end{equation}
Thus we can rewrite $H$ in the following, more concrete, form
\begin{equation} \label{2.121} 
H({\bf v}) = \int_{\Omega} |\nabla {\bf v}|^2 +2t{\bf v}\cdot \nabla \times {\bf v} -(\delta_{t} -t^2+2t^2\sin^2 \theta)|{\bf v}|^2\,dx.
\end{equation}
Firstly we infer that $0<\delta_{t}\leqslant \frac{\pi^2}{4}$. This is clearly true because looking at \eqref{2.104} simply shows that $\theta'^2>\delta_{t}$. Hence if $\delta_{t}> \frac{\pi^2}{4}$ then necessarily $\theta(1)>\frac{\pi}{2}$. It is also an elementary fact that $|\nabla \times {\bf v}| \leqslant \sqrt{2}|\nabla {\bf v}|$, so 
\begin{equation} \label{2.122}
\begin{split}
H({\bf v}) & \geqslant  \int_{\Omega} |\nabla {\bf v}|^2 +2t{\bf v}\cdot \nabla \times {\bf v} - \left(\frac{\pi^2}{4}+t^2\right)|{\bf v}|^2\,dx \\
        & \geqslant  \int_{\Omega} |\nabla {\bf v}|^2 -2t|{\bf v}||\nabla \times {\bf v}| - \left(\frac{\pi^2}{4}+t^2\right)|{\bf v}|^2\,dx \\
  & \geqslant  \int_{\Omega} |\nabla {\bf v}|^2 -2t\sqrt{2}|{\bf v}||\nabla {\bf v}| - \left(\frac{\pi^2}{4}+t^2\right)|{\bf v}|^2\,dx.
\end{split}
\end{equation}
From here we can use Cauchy's inequality, $ab\leqslant \frac{1}{2}(\epsilon a^2 +\frac{b^2}{\epsilon})$, with $\epsilon = 4t$ to obtain 
\begin{equation} \label{2.123}
 H({\bf v}) \geqslant \int_{\Omega} |\nabla {\bf v}|^2 \left(1-\frac{1}{2\sqrt{2}}\right)-\left(\frac{\pi^2}{4}+t^2+4\sqrt{2}t^2\right)|{\bf v}|^2\,dx.
\end{equation}
As can be seen by the argument below, $\epsilon=4t$ is not the optimal value to use in Cauchy's inequality for proving the global minimum property for the widest range of $t$ values. A simple calculation shows that $\epsilon = \left( \frac{4\sqrt{2}}{3} + \frac{2\sqrt{11}}{3} \right) t  \approx 4.09 t$ is in fact the optimal constant but we use $4t$ for ease here. Having established \eqref{2.123} we apply the Poincar\'{e} inequality with the sharp constant $\pi^2$, to deduce that 
\begin{equation}\label{2.126}
\begin{split}
 H({\bf v}) &\geqslant \left[ \pi^2\left( 1-\frac{1}{2\sqrt{2}}\right) -\left(\frac{\pi^2}{4}+t^2+4\sqrt{2}t^2\right)\right] \int_{\Omega} |{\bf v}|^2\,dx\\
 &=:\gamma_{t}\int_{\Omega} |{\bf v}|^2\,dx\\
 &\approx \left(3.91-6.65t^2\right)\int_{\Omega}|{\bf v}|^2\,dx.
\end{split}
\end{equation}

\subsubsection*{Conclusion}

We are now in a position to conclude the proof by drawing the two strands together. By taking an arbitrary ${\bf n}\in \mathcal{A}$ we know that 
\begin{equation}\label{2.127}
 I({\bf n})-I({\bf n}^*)=H({\bf n}-{\bf n}^*)\geqslant\gamma_{t}\int_{\Omega} |{\bf n}-{\bf n}^*|^2\,dx.
\end{equation}
The definition of $\gamma_{t}$ shows that if $t\leqslant 0.766$ then $\gamma_{t}>0$. This means that there exists some $t^*>0$ such that if $t<t^*$, ${\bf n}^*$ is indeed the unique global minimiser of the problem.

\end{proof}

Theorem \ref{2:theorem_global_min} is a very interesting, if slightly unsurprising, result. For nematic liquid crystals, using the one constant approximation results in the very simple Lagrangian $|\nabla \vn|^2$. Clearly if we studied the nematic problem in isolation with the same boundary conditions we would be able to conclude that the function 
\begin{equation}\label{2.128}
 \left( \begin{array}{c} \cos\left( \frac{\pi z}{2} \right) \\ 0 \\ \sin \left( \frac{\pi z}{2} \right) \end{array} \right),
\end{equation}
would be the unique global minimum. Theorem \ref{2:theorem_global_min} proves that this behaviour propagates into cholesteric liquid crystals for long pitch lengths. This is the first time that cholesteric liquid crystals with long pitch lengths have been analytically shown to have similar qualitative behaviour to nematic liquid crystals. In the final result of this section we prove an interesting result which shows that the combination of the Lagrangian our particular set of admissible functions $\mathcal{A}$, \eqref{1.3}, removes any contribution of the saddle splay term.

\begin{proposition}\label{3: prop periodicity}

Let $\vn \in \mathcal{A}$. Then
\begin{equation}\label{eq:5.25}
 \int_{\Omega} {\rm{tr}}\left(\nabla \vn^2\right) - \left( \nabla \cdot \vn \right)^2 \, dx=0.
\end{equation}

\end{proposition}

\begin{proof}

To avoid technical issues over density of smooth functions in Sobolev spaces between manifolds we will show the slightly stronger statement that the following integral is zero
\begin{equation} \label{eq:5.26}
 J({\bf n}):=\int_{\Omega}{\rm{tr}}\left(\nabla \vn^2\right) - \left( \nabla \cdot \vn \right)^2\,dx,
\end{equation}
for all unconstrained functions
\begin{equation}\label{eq:5.26b}
 {\bf n}\in \mathcal{A}':=\left\{\,\,\left. {\bf v} \in W^{1,2}\left(\Omega,\mathbb{R}^3\right)\, \right|\,{\bf v}|_{z=0}= {\bf e}_1,\,\,{\bf v}|_{z=1}= {\bf e}_3,\,\,{\bf v}|_{x=-L_1}=
{\bf v}|_{x=L_1} ,\,\,{\bf v}|_{y=-L_2}= {\bf v}|_{y=L_2}
\,\,\right\}.
\end{equation}
 First we look to show that expression \eqref{eq:5.26} is zero for a dense subset $\mathcal{B}$, of $\mathcal{A}'$.
\begin{equation} \label{eq:5.27}
 \mathcal{B}:= \left\{ \left. {\bf v}\in C^{\infty}\left(\Omega,\mathbb{R}^3\right) \,\right|\,{\bf v}|_{z=0}={\bf e}_1,\,\,{\bf v}|_{z=1}={\bf e}_3,\,\,{\bf v}|_{x=-L_1}={\bf v}|_{x=L_1},\,\, {\bf v}|_{y=-L_2}={\bf v}|_{y=L_2}\right\}.
\end{equation}
So we take an arbitrary element ${\bf v}\in \mathcal{B}$ and we observe that since we are in the special case of a rectangular domain, the periodic boundary conditions imply that for 
arbitrary $y\in(-L_2,L_2)$, $z\in (0,1)$ (differentiability up to the boundary is ensured by \cite[p.705]{ball1984differentiability}.)
\begin{equation} \label{eq:5.28}
 \frac{\partial {\bf v}}{\partial y} (-L_1,y,z)= \lim_{h\rightarrow 0} \frac{{\bf v}(-L_1,y+h,z)-{\bf v}(-L_1,y,z)}{h} = \lim_{h\rightarrow 0} \frac{{\bf v}(L_1,y+h,z)-{\bf v}(L_1,y,z)}{h} =  
\frac{\partial {\bf v}}{\partial y} (L_1,y,z).
\end{equation}
By very similar arguments we conclude the following set of equalities
\begin{equation} \label{eq:5.29}
 \begin{split} \left.\frac{\partial {\bf v}}{\partial y}\right|_{x=-L_1}=\left. \frac{\partial {\bf v}}{\partial y}\right|_{x=L_1}, \quad &
\left. \frac{\partial {\bf v}}{\partial z}\right|_{x=-L_1}= \left.\frac{\partial {\bf v}}{\partial z}\right|_{x=L_1}, \\ 
 \left. \frac{\partial {\bf v}}{\partial x}\right|_{y=-L_2}=\left. \frac{\partial {\bf v}}{\partial x}\right|_{y=L_2},\quad & \left.\frac{\partial {\bf v}}{\partial z}\right|_{y=-L_2}=\left. 
\frac{\partial {\bf v}}{\partial z}\right|_{y=L_2}.
 \end{split}
\end{equation}
This puts us in a position where we can calculate the integral; we rearrange the saddle-splay term into a divergence and use the divergence theorem to find
\begin{equation} \label{eq:5.30}
\begin{split} 
J({\bf v}) =& \int_\Omega {\rm{tr}}\left(\nabla \vv^2\right) - \left( \nabla \cdot \vv \right)^2\, dx \\
=& \int_\Omega \nabla \cdot \left( (\vv \cdot \nabla)\vv-(\nabla \cdot \vv)\vv \right) \, dx \\
=&\int_{\partial\Omega} (({\bf v}\cdot \nabla){\bf v}-(\nabla \cdot {\bf v}){\bf v})\cdot {\bf N} dS \\
              =&\int_{x=-L_1} (({\bf v}\cdot \nabla){\bf v}-(\nabla \cdot {\bf v}){\bf v})\cdot (-{\bf e}_1) +
\int_{x=L_1} (({\bf v}\cdot \nabla){\bf v}-(\nabla \cdot {\bf v}){\bf v})\cdot {\bf e}_1 \\
&+\int_{y=-L_2} (({\bf v}\cdot \nabla){\bf v}-(\nabla \cdot {\bf v}){\bf v})\cdot (-{\bf e}_2)+
\int_{y=L_2} (({\bf v}\cdot \nabla){\bf v}-(\nabla \cdot {\bf v}){\bf v})\cdot {\bf e}_2  \\
&+\int_{z=0} (({\bf v}\cdot \nabla){\bf v}-(\nabla \cdot {\bf v}){\bf v})\cdot (-{\bf e}_3)+
\int_{z=1}(({\bf v}\cdot \nabla){\bf v}-(\nabla \cdot {\bf v}){\bf v})\cdot {\bf e}_3.
\end{split} 
\end{equation}
Where, as usual, ${\bf e}_i$ represent the usual $\mathbb{R}^3$ basis vectors. In order to deduce that each of the above lines in the previous equation are equal to zero we note that 
\begin{equation} \label{eq:5.31}
\begin{split} (({\bf v}\cdot \nabla){\bf v}-(\nabla \cdot {\bf v}){\bf v})\cdot {\bf e}_1 = \vv_2 \vv_{1,2}+\vv_3 \vv_{1,3} - \vv_1 \vv_{2,2} - \vv_1 \vv_{3,3} \\ 
       (({\bf v}\cdot \nabla){\bf v} - (\nabla \cdot {\bf v}){\bf v}) \cdot {\bf e}_2 = \vv_1 \vv_{2,1}+\vv_3 \vv_{2,3} - \vv_2 \vv_{1,1} - \vv_2 \vv_{3,3} \\ 
(({\bf v}\cdot \nabla){\bf v}-(\nabla \cdot {\bf v}){\bf v})\cdot {\bf e}_3 = \vv_1 \vv_{3,1}+\vv_2 \vv_{3,2} - \vv_3 \vv_{1,1} - \vv_3 \vv_{2,2}.
\end{split}
\end{equation}
The first two equalities in \eqref{eq:5.31} show that the first two lines of \eqref{eq:5.30} are indeed zero by utilising \eqref{eq:5.29}. Finally we note that since ${\bf v}$ is prescribed and constant in the $z=0,1$ faces we must have that 
\begin{equation} \label{eq:5.32}
\left. \frac{\partial {\bf v}}{\partial x}\right|_{z=0}=\left. \frac{\partial {\bf v}}{\partial y} \right|_{z=0}=\left.\frac{\partial {\bf v}}{\partial x}\right|_{z=1}=
\left.\frac{\partial {\bf v}}{\partial y}\right|_{z=1}=0,
\end{equation}
meaning that the final line in \eqref{eq:5.30} is zero as both of the integrands are identically zero. So we have that
\begin{equation}\label{eq:5.32b}
 J({\bf v})=0\quad\forall \, {\bf v}\in \mathcal{B}.
\end{equation}
So now if we take an arbitrary element ${\bf n}\in \mathcal{A}.$, the density of $\mathcal{B}$ implies that 
\begin{equation} \label{eq:5.33}
 \exists ({\bf v}^j) \subset \mathcal{B}\,\,\text{such}\,\,\text{that}\,\,||{\bf v}^j - {\bf n}||_{1,2} \longrightarrow 0 \,\,as\,\,j\rightarrow \infty.
\end{equation}
With this information it is clear when writing \eqref{eq:5.26} in component form and using the fact that 
\begin{equation}\label{eq:5.34}
f^j\rightarrow f\,\, \text{in}\,\, L^2 \quad \text{and} \quad  g^j\rightarrow g\,\, \text{in}\,\, L^2 \quad \Rightarrow \quad f^j g^j \rightarrow fg \quad in\quad L^1,
\end{equation}
we deduce the convergence of the integral values so that
\begin{equation} \label{eq:5.35}
 0=J({\bf v}^j)\longrightarrow J({\bf n}).
\end{equation}
Therefore $J({\bf n})=0$ for all ${\bf n}\in\mathcal{A}'$ and $\mathcal{A}\subset\mathcal{A}'$, so the statement is proved.

\end{proof}

This proposition indicates that there may be a subtle flaw somewhere in our modelling of the cholesteric liquid crystal problem. We know that there are periodic structures that can form in cholesteric liquid crystals which make good energetic use of the saddle-splay term (blue phases for example). However the function space choice of $W^{1,2}\left(\Omega ,\mathbb{S}^2 \right)$ would not allow us to be able to predict such a structure. This issue will be explored in depth a subsequent paper \cite{bedford2014function} where we show that considering director fields in SBV$(\Omega ,\mathbb{S}^2)$ more easily allows for the prediction of multi-dimensional cholesteric structures.

%
%
%
%
%
%

\section{Alternative problem}\label{section: second problem}

Up to this point we have exclusively focused our attentions on the problem with perpendicular boundary conditions applied to the top and bottom faces of the cell. However a more commonly studied problem is that of homeotropic boundary conditions on both faces. In this final section we show that much of our analysis from Theorem \ref{2:theorem_global_min} can be easily applied to such a situation. Therefore we consider the problem of minimising
\begin{equation}\label{10.1}
 I(\vn) = \int_\Omega |\nabla \vn|^2 + 2t \vn \cdot\nabla\times \vn + t^2\, dx
\end{equation}
over 
\begin{equation}\label{10.2}
 \mathcal{A}:= \left\{ \, \left. \vn \in W^{1,2}\left( \Omega , \mathbb{S}^2 \right)\, \right| \, \left. \vn \right|_{z=0} = \left. \vn \right|_{z=1} = {\bf e}_3,\,\,\vn|_{x=-L_1} = \vn|_{x=L_1},\,\, \vn|_{y=-L_2} = \vn|_{y=L_2}\, \, \right\}.
\end{equation}
This problem is much more approachable in that for every value of $t$ we have a simple solution of the Euler-Lagrange equation. The constant function $\vn = {\bf e}_3$ satisfies 
\begin{equation}\label{10.3}
 \Delta \vn -2t\nabla \times \vn +\vn \left( |\nabla \vn|^2 + 2t\vn\cdot\nabla \times \vn \right) = 0,
\end{equation}
for any $t$. Accordingly this simplifies much of the analysis to deduce whether this a global or local minimiser. This constant state is often called the unwound cholesteric state because of its deviation from the helical configuration
\begin{equation}\label{10.4}
 \vn(z):=\left( 
 \begin{array}{c} \cos(tz) \\ \sin(tz) \\ 0 \end{array}
 \right).
\end{equation}
The stability of the unwound state has been approached both experimentally \cite{gartland2010electric}, and with simulations. However with the local minimiser results Theorems \ref{thm: weak local min} and \ref{thm: strong local min} we can precisely quantify its stability, both locally and globally, with the two results below.

\begin{theorem}\label{2: prop global min homeotropic case}
 There exists some $t^*>0$ such that whenever $t<t^*$, the function $\vn = {\bf e}_3$ is the unique global minimiser of \eqref{10.1} over \eqref{10.2}
\end{theorem}

\begin{proof}
 
 The main steps of this proof closely resemble those of the proof of Theorem \ref{2:theorem_global_min}, however this case is much simpler because our candidate minimiser is constant. Let $\vn^*:= {\bf e}_3$, and $\vn \in \mathcal{A}$. Then
 \begin{equation}\label{7.27}
  \begin{split}
   I(\vn) - I(\vn^*) & = \int_\Omega |\nabla \vn |^2 +2t \vn \cdot \nabla \times \vn \, dx \\
   & = \int_\Omega |\nabla (\vn-\vn^*)|^2 + 2t (\vn-\vn^*) \cdot \nabla \times (\vn - \vn^*)+2t \vn^* \cdot \nabla \times \vn \, dx \\
   & = \int_\Omega |\nabla (\vn-\vn^*)|^2 + 2t (\vn-\vn^*) \cdot \nabla \times (\vn - \vn^*) \, dx\\
   & =: H(\vn - \vn^*)
  \end{split}
 \end{equation}

 The final equality holds in \eqref{7.27} because of the periodicity in the $x$ and $y$ directions for $\vn$. From here we can follow the same logic as can be found from \eqref{2.120} to \eqref{2.123} to deduce that
 \begin{equation}\label{7.28}
 \begin{split}
  I(\vn)- I(\vn^*) = H(\vn-\vn^*)  &\geqslant \int_\Omega |\nabla (\vn - \vn^*)|^2 \left( 1- \frac{1}{2\sqrt{2}} \right) -4t^2\sqrt{2} |\vn-\vn^*|^2\,dx \\ 
  & \geqslant \left( \pi^2- \frac{\pi^2}{2\sqrt{2}}-4 t^2 \sqrt{2} \right) \int_\Omega |\vn -\vn^*|^2 \,dx. \\  
 \end{split}
 \end{equation}
Therefore when $t<t^*:= \frac{\pi\sqrt{2\sqrt{2}-1}}{4}\approx 1.06$, we have that 
\begin{equation}\label{7.29}
 I(\vn)-I(\vn^*) \geqslant \gamma_t \int_\Omega |\vn - \vn^*|^2\,dx
\end{equation}

for some $\gamma_t >0$. This proves the global stability of $\vn^*$.

\end{proof}

\begin{theorem}\cite{bedford2014analysis}\label{thm: e_3 stability}
Consider the variational problem as described in \eqref{10.1} and \eqref{10.2}. If $t<\pi$ then $\vn={\bf e}_3$ is a strict strong local minimiser of $I$. If $t>\pi$ then $\vn={\bf e}_3$ is not a weak local minimiser of $I$.
\end{theorem}

\begin{remark*}
 As was mentioned in \cite{bedford2014analysis} it is unknown whether the constant state ${\bf e}_3$ is a local minimiser or not at the critical value of $t=\pi$. The problem smoothly depends on the cholesteric twist $t$, however unless we can show that ${\bf e}_3$ is a global minimiser up to $t=\pi$ we know of no way to deduce its stability at the bifurcation point.
\end{remark*}

These results tie in with Theorem \ref{2:theorem_global_min} in that they too demonstrate how cholesteric liquid crystals are related to nematic liquid crystals for long pitch lengths. It is immediately obvious that if we studied nematic liquid crystals with the set of admissible functions as given in \eqref{10.2}, the unique global minimiser is $\vn = {\bf e}_3$. According to Theorem \ref{2: prop global min homeotropic case}, this property propagates into the cholesteric setting for small values of $t$. Theorem \ref{thm: e_3 stability} extends this result and shows that ${\bf e}_3$ is a strong local minimiser for the problem up to $t=\pi$, but not a weak local minimum thereafter. It is interesting to note that $\vn = {\bf e}_3$ is a global minimum at least up to $t = 1.06$ and a strong local minimiser up to $t=\pi$. This obviously allows the possibility of there being a range of $t$ values where the unwound state is a local minimiser but not a global minimiser. As far as we are aware, the theorems presented in this paper are the first to show this qualitative correspondence between nematic liquid crystals and low chirality cholesteric liquid crystals

\vspace{3mm}

To the best of our knowledge it is still an open problem to analytically prove the existence of multi-dimensional minimisers for a specific cholesteric liquid crystal problem when the director field $\vn$ lies in the space $W^{1,2}\left( \Omega ,\mathbb{S}^2 \right)$. It may indeed be possible to prove this fact, whether analytically or numerically, using Sobolev regularity director fields. However in a forthcoming paper \cite{bedford2014function}, we will show that we can predict the existence of these structures much more easily when we relax the problem so that $\vn \in SBV^2\left(\Omega,\mathbb{S}^2 \right)$. The technique has potential because it allows the Oseen-Frank theory to better respect the head-to-tail symmetry of the molecules.

\subsection*{Acknowledgements}

This work was supported by the EPSRC Science and Innovation
award to the Oxford Centre for Nonlinear PDE (EP/E035027/1). The author is supported by CASE studentship with Hewlett-Packard Limited.

\bibliographystyle{plain}
\bibliography{refs} 

\end{document}